\documentclass[12pt]{amsart}
\usepackage{graphicx}
\usepackage{amssymb}
\usepackage{cite}
\usepackage{hyperref}
\usepackage{enumerate}
\usepackage{txfonts} 
\usepackage{mathtools}
\usepackage[usenames,dvipsnames]{color}

\numberwithin{equation}{section} 

\newtheorem{theorem}{Theorem}[section]
\newtheorem{theorem*}{Theorem}
\newtheorem{lemma}[theorem]{Lemma}

\newtheorem{corollary}[theorem]{Corollary}

\newtheorem{remark}[theorem]{Remark}
\newtheorem{remark*}[theorem*]{Remark}

\def\R{{\mathbb R}}

\def\cB{{\mathcal B}}

\def\G{\Gamma}

\def\r{\rho}

\def\1{\left(}
\def\2{\right)}
\def\3{\left\{}
\def\4{\right\}}
\def\8{\infty}
\def\sm{\setminus}

\begin{document}
\title[]{Regularity for parabolic systems with critical growth in the gradient and applications}

\author[A. Banerjee]{Agnid Banerjee}
\address{Tata Institute of Fundamental Research\\
Centre For Applicable Mathematics \\ Bangalore-560065, India}\email[Agnid Banerjee]{agnidban@gmail.com}
\author[G. D\'avila]{Gonzalo D\'avila}
\address{
Departamento de Matem\'atica, Universidad T\'ecnica Federico Santa Mar\'ia \\
Casilla: v-110, Avda. Espa\~na 1680, Valpara\'iso, Chile
}
\email{gonzalo.davila@usm.cl}

\author[Y. Sire]{Yannick Sire}
\address{\noindent Department of Mathematics, Johns Hopkins University, 404 Krieger Hall, 3400 N. Charles Street, Baltimore, MD 21218, USA}

\thanks{A.B.  was  supported in part by SERB Matrix grant MTR/2018/000267 and also  by Department of Atomic Energy,  Government of India, under
project no.  12-R \& D-TFR-5.01-0520.\\
G. D. was partially supported by Fondecyt Grant No. 1190209.}

\begin{abstract}
Motivated by problems arising in geometric flows, we prove several  regularity results for systems of local and nonlocal equations, adapting to the parabolic case a neat argument due to Caffarelli.  The geometric motivation of this work comes from recent works arising in the theory harmonic maps with free boundary in particular. We prove H\"older regularity of weak solutions.  
\end{abstract}

\maketitle

\tableofcontents

\section{Introduction}

Elliptic and parabolic systems with critical growth in the gradient appear in many geometric variational problems (see for instance \cite{heleinBook,riviere}). The purpose of this paper is to develop a regularity theory for such systems, with a {\sl subcritical} condition which ensures {\sl full and not partial} regularity. We follow a nice approach provided by Caffarelli in \cite{Caffarelli-Systems} for elliptic systems. Our motivation comes from some problems recently introduced in the theory of minimal surfaces with free boundary through the study of the so-called harmonic maps with free boundary and their flows (see e.g. \cite{Millot-Sire}). As described in \cite{Millot-Sire,DLP}, harmonic maps with free boundary can be reformulated in a very natural way as harmonic maps with respect to an $\dot H^{1/2}$-energy hence leading to a nonlocal system of elliptic equations. See \cite{DR,DR2}

Regularity is a major issue in the theory of PDEs and in geometric analysis. The feature of the problems under consideration in the present paper is that the systems are {\sl critical} for the scaling, hence for the regularity, but somehow {\sl subcritical} for the underlying geometry, making full regularity possible. We are mainly concerned with problems arising in the theory of harmonic map and their heat flows. We refer the reader to the book \cite{bookLW} for an updated account on the theory. 

In order to facilitate the presentation, we will first provide a proof in the case of the following parabolic equation, arising as the heat flow of harmonic maps into Euclidean spheres
\begin{align}
u_t-\Delta u - u|Du|^2=0, \  \text{in } B_2\times(-1,1)
\end{align}

Then, we will investigate one type of nonlocal parabolic equations, namely for $s \in (0,1)$ and $n \geq 1$

\begin{align}\label{fracIntro1}
u_t+(-\Delta)^s u = u\,\mathcal B(u,u), \  \text{in } \R^n\times(-1,1)
\end{align}

where $(-\Delta)^s$ stands for the Fourier multiplier of symbol $|\xi|^{2s}$ and $\mathcal B(u,u)$ is a suitable bilinear form;

We will also consider the equation 
\begin{align}\label{fracIntro2}
(\partial_t-\Delta)^s u =0, \  \text{in } \R^n\times \R
\end{align}
where $(\partial_t-\Delta)^s$ stands for the Laplace-Fourier multiplier of symbol $(i\tau+|\xi|^2)^s$.

Geometrically, the nonlocal operator involved in \eqref{fracIntro1} come from a formulation of the heat flow of harmonic map with free boundary written in nonlocal form, namely 
$$
\partial_t u +(-\Delta)^s u \perp T_{u(x)} \mathbb S^{m-1}. 
$$
Related to another formulation of this flow is the distributional equation given by 
$$
(\partial_t-\Delta)^s u \perp T_{u(x)} \mathbb S^{m-1},
$$
motivating the study of \eqref{fracIntro2}. The previous equations have to be understood distributionally of course. Clearly, these two flows admit the same stationary solutions, which are $s-$harmonic maps into spheres, which have been studied in \cite{DR,DR2,Millot-Sire,pegon}. See also \cite{roberts} for an intrinsic version of such maps.  

To put our results in perspective, we would like to mention that as far as the regularity of heat flow of harmonic maps is concerned, a way to construct weak solutions and study their regularity is to have a suitable monotonicity formula for a Ginzburg-Landau approximation of the system (see the monograph \cite{bookLW}). At the moment such a  monotonicity formula is {\sl not } available for the operator involved in  \eqref{fracIntro1}, i.e. the usual fractional heat operator. On the other hand, it is known (see \cite{banerjeeGaro}) that such a quantity exists for \eqref{fracIntro2}. This is due to the existence of a suitable extension to the upper-half space (see \cite{CafSil,ST,NS}). The absence of monotone quantity for \eqref{fracIntro1} is actually the main motivation to investigate the regularity of this system.

\subsection*{Main results}

In this section, we state our main results. The first theorem deals with local elliptic systems arising in the heat flow of harmonic maps into spheres. 

\begin{theorem}\label{2ndorder}
Let $u$ be a weak solution in $\R^n \times(-1,1]\to\R^m$ of 
\begin{align}\label{parabolicsystem}
u_t-\Delta u - u|Du|^2=0, 
\end{align}
such that $\|u\|_\infty <1$. Then $u$ is $C^\alpha_{loc} (\R^n \times(-1,1])$ for some $\alpha \in (0,1)$.

\end{theorem}

To introduce our next result, we define first the integral formulation of the fractional laplacian. Given $s\in(0,1)$ and a vector-valued function $u$ denote the operator $(-\Delta)^s$ by (component-wise)
\[
(-\Delta)^su(x)=P.V.\int_{\R^n}\frac{u_\ell(y)-u_\ell(x)}{|x-y|^{n+2s}}\,dy,
\]
and define a bilinear form $\cB$ by
\[
\cB(u,w)=\frac{(1-s)c_n}{2}\int_{\R^n}\frac{(u(x)-u(y))\cdot(w(x)-w(y))} {|x-y|^{n+2s}}\,dy,
\]

\begin{theorem}\label{nonlocal}
Let $u$ be a weak solution in $\R^n \times(-1,1]\to\R^m$ of 
\begin{align}\label{fracparabolicsystem}
u_t+(-\Delta )^s u = u\cB(u,u), \  \text{in } \R^n \times(-1,1]
\end{align}
such that $\|u\|_\infty <1$ in $\R^n \times (-1,1)$. Then $u$ is $C^\alpha_{loc} (\R^n\times(-1,1])$ for some $\alpha \in (0,1)$. 
\end{theorem}

It will be clear from the proof of Theorem \ref{nonlocal} that one can consider more general systems involving a more general (than the fractional laplacian) integral operator $\mathcal L$ defined above, provided the assumption $\|u\|_\infty=M <1$ where $M$ is small enough depending on the ellipticity constants of $L$.

Theorem \ref{2ndorder} generalizes the original result of Caffarelli \cite{Caffarelli-Systems} whereas Theorem \ref{nonlocal} generalizes the one of Caffarelli with the second author \cite{CaD}. As was pointed out in \cite{CaD} (see also \cite{Millot-Sire} ) nonlocal systems of this type arise in the theory of fractional harmonic maps, which is for $s=1/2$ related to minimal surfaces with free boundary (see also \cite{roberts} where the techniques of the present paper apply to both regularity in the interior and up to the boundary).

We refrain here from dealing with the previous systems in full generality as far as the statements are concerned; but it is clear, as pointed in \cite{Caffarelli-Systems} how to adapt the argument to non-homogeneous equations for instance (allowing dependence in $x$) under some natural structural assumptions. Actually, an important feature of the nonlocal equations under consideration is that a standard iteration {\sl does not work} and one has to enlarge the class of systems to ensure that at every step, one can apply the main oscillation reduction lemma.

We now state our third main result, concerning the powers of the heat operator. Denote in the following $H^s=(\partial_t - \Delta)^s$. We prove: 
\begin{theorem}\label{fHarnackIntro}
Let $v \in \operatorname{Dom}(H^s)$ be a strong  supersolution of 
\[
H^s v=0,
\]
in $B_2\times(-4,0]$, $0\leq v$ in $\R^n \times \R$. Then there exists a constant $C$ such that 
\begin{align}\label{f2}
\|v\|_{L^1(U^-)}\leq C\inf\limits_{U^+} v(x,t),
\end{align}
where \[
U^-=B_{1/2}\times \left(-1, -\frac{1}{2}\right) \ \  U^+= B_{1/2} \times \left(-\frac{1}{4}, 0\right).
\]
\end{theorem}

Crucially used in the proof of Theorem \ref{nonlocal}, one needs to get for the problem under consideration a weak Harnack inequality which takes into accounts the {\sl tail} effects due to the nonlocality. For the operator $H^s$, such a weak Harnack inequality is not available. However, Theorem \ref{fHarnackIntro} provides another version of the weak Harnack inequality of independent interest and this is why, even if not leading with our techniques to the desired H\"older regularity, is of independent interest. 

The use of the operator $H^s$ is motivated by the following geometric flow  (see e.g. \cite{chen-lin} and references therein)
\begin{equation}\label{e:main}
\begin{cases}
u_t = \Delta u\text{ in }\mathbb{R}^n_+\times (0, T),\\
u(x,0,t) \in \mathbb{S}^{m-1}\text{ for all }(x,0,t)\in \partial\mathbb{R}^n_+\times (0, T),\\
\frac{\partial u}{\partial \nu}(x,0,t)\perp T_{u(x,0,t)}\mathbb{S}^{m-1}\text{ for all }(x,0,t)\in \partial\mathbb{R}^n_+\times (0, T),\\
u(\cdot, 0) = u_0\text{ in }\mathbb{R}^n_+
  \end{cases}
\end{equation}
for a function $u:\mathbb{R}^n_+\times [0, T)\to \mathbb{R}^m$. Here $u_0 :\mathbb{R}^n_+\to \mathbb{R}^m$ is a given smooth map and $\perp$ stands for orthogonality. 

Using \cite{ST}, the previous system can be reformulated as: let $u \in \text{Dom}\ (H^\frac12): \R^n \times \R \to \R^m$ be a strong solution  of 
\begin{align}\label{fracparabolicsystem2}
(\partial_t-\Delta)^\frac12 u = u\mathcal C(u,u), \  \text{in } \R^n \times \R
\end{align}
where
$$
\mathcal C(u,u)=\frac12\int_{0}^{+\infty}\int_{\mathbb{R}^n} |u(x,t) - u(x-z,t-\tau)|^2\frac{e^{-\frac{|z|^2}{4\tau}}}{(4 \pi)^{n/2}|\Gamma(-s)| \tau^{n/2 + 1 +\frac12}}dzd\tau
$$

In the present work, we consider only H\"older regularity. Actually, as far as Theorem \ref{2ndorder} is concerned, it is classical (see \cite{bookLW}), that such solutions are actually $C^\infty$ in space and time. The case of the parabolic nonlocal equations is more technical to consider. Indeed, due to the nonlocality of the problem, the standard compactness or potential-theoretic arguments need to be non trivially adapted due to the criticality of the right hand side of the equations. We postpone to a forthcoming work higher order regularity for such equations with critical dependence on the gradient, since it might be of independent interest.

\section{The second order case: Proof of Theorem \ref{2ndorder}}

Let $u:B_2\times(-1,1]\to\R^m$ be a bounded weak solution of
\begin{align*}
u_t-\Delta u - u|Du|^2=0, \  \text{in } B_2\times(-1,1)
\end{align*}
We are interested in proving interior H\"older regularity. For the rest of the note we will write
\[
M=\|u\|_{L^\infty(B_2\times(-1,1])}.
\]

In order to prove the interior H\"older estimates we need a weak Harnack inequality for supersolutions of the scalar heat equation. For completeness we state the result as it was presented by N. Trudinger (Theorem 1.2 in \cite{Tr}).

Denote $K_{x_0}(\rho)$ cube of center $x_0$ and side $\rho$. Furthermore, given $x_0\in\R^n$, $t_0\in\R$, $\rho>0$ and $\tau>0$ denote the rectangle $R$ by
\[
R=K_{x_0}(\rho)\times(t_0-\tau\rho^2,t_0)
\]
and the subrectangles
\begin{align*}
R^+&=K_{x_0}(\rho')\times (t_0-\tau_3\rho^2, t_0-\tau_2\rho^2)\\
R^-&=K_{x_0}(\rho')\times (t_0-\tau_1\rho^2, t_0-\tau_0\rho^2)\\
R^*&=K_{x_0}(\rho'')\times (t_0-\tau\rho^2, t_0-\tau_2\rho^2),
\end{align*}
where $0<\rho'\leq\rho''$, $0<\tau_0<\tau_1<\tau_2<\tau_3<\tau$.
\begin{theorem}[Theorem 1.2 in \cite{Tr}]\label{weakHarnack}
Let $v$ be a weak supersolution of 
\[
\Delta v-v_t=0,
\]
in $R$, $0\leq v\leq M$ in $R$. Then there exists a constant $C$ such that 
\begin{align}\label{Harnackestimate}
\frac{1}{\r^{n+2}}\|v\|_{L^1(R^*)}\leq C\min\limits_{R^-} v(x,t).
\end{align}
\end{theorem}
\begin{remark}
Theorem \ref{weakHarnack} was proved in \cite{Tr} for more general divergence operators.
\end{remark}
Note also that Theorem \ref{weakHarnack} can be generalized for sets of the form $B_\rho(x_0)\times I$, where $B_\rho(x_0)$ is the ball of center $x_0$ and radius $\rho$ and $I$ is a time interval. Moreover by taking $x_0=0$, $t_0=1$, $\rho=\rho''=2$, $\rho'=1$, $\tau=1/2$, $\tau_2=1/3$, $\tau_1=1/4$ and $\tau_0=1/8$ we can rewrite the estimate \eqref{Harnackestimate} as
\begin{align}\label{Harnackestimate1}
\fint_{B_1\times(-1,-1/3)} v(x,t)dxdt\leq C\min_{B\times (0,1/2)}v,
\end{align}
where we denote
\[
\fint_{B\times I}v(x,t) dxdt=\frac{1}{|B\times I|} \int v(x,t) dxdt
\]

Let us denote $B^*=B_1\times(-1,-1/3)$ and $B^-= B\times (0,1/2)$. Now we are in shape to proceed as in \cite{Caffarelli-Systems}.
\begin{lemma}\label{1stcontrol}
Let $u$ be a weak solution to 
\begin{equation}\label{fd1}
u_t - \Delta u= f(x, u, \nabla u)
\end{equation}

 in $B_2(0)\times (-1,1)$ such that  $||u||_{L^{\infty}} =M<1$. Moreover $f$ satisfies 
 \begin{equation}\label{fbd}
 \begin{cases}
 |f(x,u, \nabla u)| \leq |\nabla u|^2 
 \\
 |u f(x, u, \nabla u)|\leq l |\nabla u|^2
 \end{cases}
 \end{equation}
 for some $l <1$. 
 
  Then there exists a constant $0<\delta(l)<1$ (monotone decreasing in $l$)  such that 
\[
u(B^-)\subset B_{M(1-\delta)}(\delta\bar u),
\]
where
\[
\bar u=\frac{1}{|B^*|}\int_{B^*}u dxdt=\fint_{B^*}udxdt.
\]
\end{lemma}
\begin{proof}
As in \cite{Caffarelli-Systems} the strategy revolves in using $|u|^2$ as a supersolution of a linear scalar equation. 

First note that using \eqref{fbd}, 
\begin{align*}
\left(\frac{\partial}{\partial t}-\Delta\right)(1/2|u|^2)&=u\cdot u_t-u\cdot\Delta u -|\nabla u|^2\\
&\leq (l-1)|\nabla u |^2.
\end{align*}
Now, let $\xi\in\R^m$ with $|\xi|\leq (1-l)$ and note that
\begin{align*}
\left(\frac{\partial}{\partial t}-\Delta\right)(\xi\cdot u)&=\xi\cdot f(x, u, \nabla u)\\
&\leq (1-l)|\nabla u|^2.
\end{align*}
The previous inequality leads to
 \[
\left(\frac{\partial}{\partial t}-\Delta\right)(-1/2|u|^2-\xi\cdot u)\geq 0
\]
Recall now that $u$ is bounded by $M$, therefore  
\[
h(x)=\frac{1}{2}M^2+(1-l)M-\frac{1}{2}|u|^2-\xi\cdot u,
\]
is nonnegative supersolution of the scalar heat equation in $B_2\times(-1,1)$ and therefore we are shape to apply Theorem \ref{weakHarnack}. Note also that
\begin{align*}
\bar h:=\fint_{B^*} h(x,t)&=\frac{1}{2}M^2+(1-l)M-\frac{1}{2}\fint_{B^*}|u|^2dxdt-\xi\cdot\fint_{B^*}udxdt\\
&\geq (1-l)M-\xi\cdot \bar u.
\end{align*} 
We apply now the weak Harnack inequality to $h$ (see \eqref{Harnackestimate1}) to conclude that there exists a constant $C$  such that
\begin{align*}
Ch(x)&\geq \bar h\\
&\geq (1-l)M-\xi\cdot \bar u
\end{align*}
for $x\in B^-$, or equivalently
\[
h(x)\geq c((1-l)M-\xi\cdot \bar u),
\]
for a universal constant $c$ that we can assume smaller than 1 ($c<1$). Now, using the definition of $h$ we get
\begin{align}\label{control1}
\frac{M^2-|u|^2}{2}+(1-l)M-\xi\cdot u\geq c ((1-l)M-\xi\cdot\bar u). 
\end{align}
Fix $(x,t)\in B^-$ and take $\xi$ in the direction of $u(x,t)$, more precisely let $\xi$ be given by
\[
\xi=(1-l)\frac{u}{|u|},
\]
and denote by $\theta$ the angle between $u$ and $\bar u$ and let $r=|u|/M$. Note that with this selection of parameters we have
\[
(1-l)M-\xi\cdot u=M(1-l)(1-r)
\]
and therefore, from \eqref{control1}, we get
\begin{align*}
M(1-r)\left( \frac{1}{2}(M+|u|)+(1-l)\right)&\geq c\left((1-l)M-\xi\cdot\bar u\right)\\
&=cM(1-l)\left(1-\frac{\cos(\theta)|\bar u|}{M}\right)
\end{align*}
which gives us the control on $r$
\begin{align}\label{control2}
1-r\geq c_1\left(1-\frac{|\bar u|}{M}\cos\theta\right).
\end{align}
 Therefore by multiplying \eqref{control2} by $r$ and adding afterwards $1-r$ we arrive to 
\[
1-r^2\geq c_1\left(1-r\frac{|\bar u|}{M}\cos\theta\right),
\]
which is equivalent to 
\[
r^2-c_1r\frac{|\bar u|}{M}\cos\theta\leq 1-c_1.
\]
Note now that $\bar u/M\leq 1$, therefore from the previous inequality we get
\[
r^2-c_1r\frac{|\bar u|}{M}\cos\theta+\left(\frac{1}{2}c_1\frac{\bar u}{M}\right)^2\leq 1-c_1 +\left(\frac{1}{2}c_1\right)^2,
\]
and by picking $\delta=1/2c_1$ we conclude
\[
|u-\delta\bar u|^2\leq M^2(1-\delta)^2.
\]
which finishes the proof.
\end{proof}
The previous lemma states that $u$ maps $B^*=B_1\times(0,1/2)$ to a ball of strictly smaller radius than 1 and center shifted toward $\bar u$. This result will allow us to control the oscillation of the function via dilations. 

A direct consequence of the previous lemma is the following corollary.
\begin{corollary}\label{corocontrol1}
Let $u$ be as in Lemma \ref{control1}. Then there exist a sequence of points $\{\rho_k\}$ and radii $\{M_k\}$ such that 
\begin{itemize}
\item[i.-] $M_k\leq M(1-\delta)^k$.
\item[ii.-] $|\rho_k|+M_k\leq M$.
\item[iii.-] $u(B_{2^{-k}}(0)\times I_k)\subset B_{M_k}(\rho_k) $, where $I_k=(a_k,b_k)$ is an interval given by the induction process
\begin{align*}
a_0&=0, \ b_0=1/2, \\
a_k&=\frac{a_{k-1}+b_{k-1}}{2}, \ b_k=\frac{a_{k-1}+b_{k-1}}{2}+\frac{b_{k-1}-a_{k-1}}{4}.
\end{align*}
\end{itemize}
\end{corollary}
\begin{proof}
We proceed by induction on $k$. Note the case $k=0$ is just Lemma \ref{control1} with $\rho_0=\delta\bar u$ and $M_0=M(1-\delta)$. Let $u_k=u(2^{-k}x,2^{-2k}(t-a_k) )-\rho_{k-1}$ and assume the result holds up to $k-1$. In order to apply Lemma \ref{control1} to $u_k$ in $B_2\times(-1,1)$ we first note that 
\[
\|u_k\|_{L^\infty(B_2(0)\times(-1,1))}\leq M_{k-1}
\]
and solves an equation of the type \eqref{fd1} with $l=M$. 
 Therefore we can apply Lemma \ref{control1} to $u_k$, which finishes the proof by letting $M_k=M_{k-1}(1-\delta)$, $\rho_k=\rho_{k-1}+\delta\bar u_k$ and observing that 
\[
M_k+\rho_k=M_{k-1}+\rho_{k-1}+\delta(\bar u_k-M_{k-1})\leq M
\]
\end{proof}

We point out that Lemma \ref{corocontrol1} gives us a control on the oscillation of the solution of \eqref{parabolicsystem}. More precisely we get that $u$ is H\"older continuous at the point $(0, t^*)$, where $t^*$ is given by $\cap I_k$. Furthermore, since the equation is translation invariant one can repeat the argument to conclude that the solution is H\"older in the interior of the domain. The higher order regularity follows now from standard techniques, see e.g. \cite{HW,bookLW, CWY}.

\section{The nonlocal cases: Proof of the H\"older regularity}\label{SecparNL}

\subsection{The case of the fractional heat operator}
We begin this section with some preliminary facts. Recall first 
\[
\cB(u,w)=\frac{(1-s)c_n}{2}\int_{\R^n}\frac{(u(x)-u(y))\cdot(w(x)-w(y))} {|x-y|^{n+2s}}\,dy. 
\] 
Let $v:\R^n\to\R$ be a smooth bounded function, then we claim that
\[
-L v^2(x)=-2v(x)L_Kv(x)-2\cB(v,v),
\]
A direct algebraic manipulation gives 
\[
(-\Delta)^su^2=2u(-\Delta)^su+2\cB(u),
\]
where we are denoting $\cB(u,u)=\cB(u)$. Note that $\cB(u)$ has the same scaling as the fractional laplacian, that is, for $\lambda\in\R$ we have $\cB(u_\lambda))(x)=\lambda^{2s}\cB(u)(\lambda x)$, where $u_\lambda(x)=u(\lambda x)$. Moreover $\cB(u,v)$ plays the role of $Du\cdot Dv$.

We are interested in studying interior regularity for solutions of the fractional parabolic system 
\begin{align}\label{frac1}
u_t+(-\Delta)^s u = u\cB(u)
\end{align}
For the time independent case the previous system was studied in \cite{CaD} based on the techniques developed in \cite{Caffarelli-Systems}. One of the key elements in the proof is the weak Harnack inequality.

We point out that a fractional version of \eqref{weakHarnack} exists for weak and viscosity supersolutions of the fractional heat equation (see for example \cite{FK} for weak solutions and \cite{CD} for viscosity ones). In the context of the present work, we however need the following weak Harnack inequality with "tail" as established in \cite{St} which doesn't require the supersolution to be globally positive (See Theorem 1.2 in \cite{St}).  \begin{theorem}\label{fracweakHarnack}
Let $v$ be a weak supersolution of 
\[
v_t+(-\Delta)^s v=0\ \text{in $\R^n \times (t_0 - 4r^{2s}, t_0]$}
\]
such that $v \geq 0$ in  $\Omega=B_R(x_0)\times(t_0-2r^{2s},t_0]$.  Let $0<r<R/2$. Then there exists a constant $C$ such that 
\begin{align}\label{fracHarnackestimate}
\fint_{U^-}vdxdt\leq C\inf\limits_{U^+} v(x,t)+C\left(\frac{r}{R}\right)^{2s}\text{Tail}_\infty(v^-;x_0,R,t_0-2r^{2s}, t_0),
\end{align}
where \[
U^-=B_r(x_0)\times(t_0-2r^{2s},t_0-r^{2s}], \ \  U^+=B_r(x_0)\times(t_0- \frac{r^{2s}}{2},t_0].
\]
and 
\[
\text{Tail}_\infty(v;x_0,R,t_1, t_2)=R^{2s}\sup\limits_{t_1<t<t_2}\int_{\R^n\sm B_R(x_0)}\frac{|v(x,t)|dxdt}{|x-x_0|^{n+2s}}
\]
\end{theorem}
In what follows we will always take $(x_0, t_0) =0$.

 
In order to prove the next result, we study an auxiliary problem. Let us consider the parabolic system
\begin{align}\label{auxcut}
\left(\frac{\partial}{\partial t}+(-\Delta)^s\right)w(x,t)=f(x,w,\cB(w)),
\end{align}
and assume
\begin{itemize}
\item[(H1.1)] {\it Small $2s$ growth:} There exists a  constant $a$ such that
\[
|f(x,w(x),\cB(w(x),w(x)))|\leq a\cB(w(x),w(x))
\]
for all maps $w:\R^n\to\R^m$.
\item[(H1.2)] There exists another constant   $l$ such that
\[
w(x)\cdot f(x,w(x),\cB(w(x),w(x)))\leq l\cB(w(x),w(x)),
\]
 for all  maps $w:\R^n\to\R^m$.
\end{itemize}
From now on and without loss of generality we assume that $a=1$.
\begin{lemma}\label{1stcontrolfraccorrected}
Let $w$ be a bounded weak solution to \eqref{auxcut} in $\R^n \times (-1,1)$ with $f$ satisfying (H1.1), (H1.2).  Assume that $l<1$, and also that   $M=\|w\|_{L^\infty(\R^n \times (-1,1))} <1$. Then there exist constants $0<\delta(l)<1$ and $\tau$ such that for  any radius $r< \frac{1}{2}$  we have that
\[
w(U^+)\subset B_{M(1-\delta)}(\delta\bar w),
\]
where
\[
\bar w=\frac{1}{|U^-|}\int_{U^-}u dxdt=\fint_{U^-}wdxdt.
\]
Furthermore $\delta$ is monotone decreasing in $l$. 
\end{lemma}

\begin{proof}

Define 
\[
h(x,t)=\frac{1}{2}M^2+(1-l)M-\frac{1}{2}|w|^2-\xi\cdot w,
\] 
and observe that $h$ is a nonnegative supersolution  in $\R^n \times(-1,1)$. Indeed note first that
\begin{align*}
\left(\frac{\partial}{\partial t}+(-\Delta)^s\right)(1/2|w|^2)&=w\cdot w_t+w\cdot(-\Delta)^s w - \cB(w)\\
&=w\cdot f(x,w,\cB(w))-\cB(w)\\
&\leq (l-1)\cB(w).
\end{align*}
Let $\xi\in\R^m$ with $|\xi|\leq (1-l)$ and note that
\begin{align*}
\left(\frac{\partial}{\partial t}+(-\Delta)^s\right)(\xi\cdot w)&=\xi\cdot f(x,w,\cB(w)))\\
&\leq |\xi|\cB(w)\\
&\leq (1-l)\cB(w).
\end{align*}
 With the previous computations it is direct to verify that 
\[
\left(\frac{\partial}{\partial t}+(-\Delta)^s\right)h\geq 0\quad \text{in }\R^n \times(-1,1),
\]
and $h\geq 0$ in $\R^n\times[-1,1]$. Note also that
\begin{align*}
\bar h:=\fint_{U^-} h(x,t)&=\frac{1}{2}M^2+(1-l)M-\frac{1}{2}\fint_{U^-}|w|^2dxdt-\xi\cdot\fint_{U^-}wdxdt \\
&\geq (1-l)M-\xi\cdot \bar w,
\end{align*} 
therefore, we deduce
\[
\bar h\geq (1-l)M-\xi\cdot \bar w.
\] 

We can now apply the weak harnack as in Lemma \ref{fracweakHarnack} ( note that the $\text{Tail}_{\infty} (h^{-}, \cdot) \equiv 0$ in this case)  and repeat the arguments as in the proof of  Lemma \ref{1stcontrol} to obtain the desired conclusion. 

\end{proof}


We now establish the following oscillation decay for solutions to \eqref{frac1}.  Note that in the nonlocal framework, the classical oscillation decay method does not work because of the "tail" effects. To overcome this issue, this is precisely why we need to have at our disposal  a Harnack inequality with a  tail as in Theorem \ref{fracweakHarnack} above.  As the reader will see, the proof of the  oscillation decay  at successive steps as in  Lemma \ref{corocontract} below  requires  certain delicate adaptation of the local arguments  precisely due to the fact that the rescaled functions are not globally non-negative and thus the "tails" have to be appropriately estimated.

\begin{lemma}\label{corocontract}
Let $u$ be a solution to \eqref{frac1} and let $||u||_{L^{\infty}(\R^n \times (-1, 1))}= M<1$. There there exists $\alpha, r \in (0,1)$ and a sequence $\{\rho_k\}$  such that 
\begin{itemize}
\item[i.-]$M_k\leq M r^{k\alpha}$,
\item[ii.-] $|\rho_{k}|+M_{k-1}\leq M$,
\item[iii.-] $u(B_{r^k} \times (- r^{2sk}, 0)) \subset B_{M_k}(\rho_k)$,
\item[iv.-] $|\rho_k - \rho_i| \leq C_0 Mr^{i\alpha}$, for all $i < k$,
\end{itemize}
where $C_0$ is some universal constant.

\end{lemma}
\begin{proof}

We will prove the result by induction. For the initial step we apply Lemma \ref{1stcontrolfraccorrected} for which any choice of $r<1/2$ works.  We now assume that the hypothesis holds up to some $k$ for a choice of $r$ to be determined below ( see \eqref{rc}). We then let
\[
u_{k}= u(r^k x, r^{2sk} t) -\rho_k.
\]
Consequently we have that  $||u||_{L^{\infty}(B_1 \times (-1, 0]) } \leq M_k$.  We also note that $$||u_k||_{L^{\infty}(\R^n \times (-1, 0]))} \leq C_1M$$ for some $C_1$ independent of $k$.  Finally we observe that $u_k$ solves an equation of the type  \eqref{auxcut} with $l=M$.  Subsequently, letting
\[
h(x,t)= \frac{1}{2}  \tilde M_k^2 - (1-l) \tilde M_k - \frac{1}{2} u_k^2 - \xi \cdot u_k,\ \text{$|\xi| \leq 1-l$}
\]
with $\tilde M_k = M r^{k\alpha}$,we note as in the proof of Lemma \ref{1stcontrolfraccorrected}, $h$ is a supersolution  globally to $\partial_t + (-\Delta)^s$ and $h \geq 0$ in $B_1 \times (-1, 0]$.
  Thus by applying the weak harnack inequality to $h$ we obtain,
\begin{equation}\label{h1}
Ch(x,t) + C r^{2s} \text{Tail}_{\infty} (h^{-}, 0, 1, -r^{2s}, 0) \geq (1-l) \tilde M_k + \xi. \bar u_k
\end{equation} 
for all $(x,t) \in B_{r} \times (-r^{2s}, 0]$. Here $\bar u_k = \fint_{B_r \times (-2r^{2s}, -\frac{5}{4} r^{2s} )} u_k$. Now by using the  induction hypothesis and  iv) which holds upto $k$, we note that 
\begin{equation}\label{k0}
||u_k||_{L^{\infty}(B_{r^{i-k}} \times (-r^{2(i-k)s}, 0])}  \leq ||u- \rho_i||_{L^{\infty}(B_{r^i} \times (-r^{2is}, 0])} + |\rho_k - \rho_i| \leq C M r^{i \alpha}
\end{equation}
for all $i \leq k$.   Using \eqref{k0}, change of variable  and by summing  the integral involving the tail over dyadic regions  as in  the proof of Lemma 5.1 in \cite{DKP}, we obtain that  the following holds, 
\begin{align}\label{k1}
& \text{Tail}_{\infty} (h^{-}, 0, 1, -r^{2s}, 0) \leq CM \sum_{i=0}^k \frac{r^{2ks}}{r^{2is}} r^{i\alpha} \leq  CM \sum_{i=0}^k r^{k\alpha}  (\frac{r^k}{r^i})^{2s- \alpha}\\
& \leq  C M r^{k\alpha} \sum_{i=0}^{\infty} r^{i\alpha} \leq  \tilde C M r^{k\alpha}\notag
\end{align}
provided we choose $\alpha < 2s$.  Note that $\tilde C$ is independent of $r$ as long as $r < 1/2$.
 Substituting the estimate in \eqref{k1}  into \eqref{h1}, we obtain that the following holds for a new $C$,
\begin{equation}\label{h2}
Ch(x,t) + C M r^{2s} r^{k\alpha}  \geq (1-l) \tilde M_k + \xi. \bar u_k
\end{equation}
At this point, we choose $r$ small enough such that  
\begin{equation}\label{rc}
Cr^{2s} < \frac{1-l}{2}
\end{equation}
We thus obtain from \eqref{h2} using \eqref{rc} that
\begin{equation}\label{h4}
Ch(x,t)  \geq \frac{1-l}{2}  \tilde M_k + \xi. \bar u_k
\end{equation}
Now by repeating the arguments as in the proof of Lemma \ref{1stcontrol} we obtain for some $\delta >0$  depending on $\frac{1-l}{2}$ that the following holds,
\begin{equation}\label{y1}
|u_k -  \delta \bar u_k|^2 \leq \tilde M_k^2(1-\delta)^2
\end{equation}
in $B_r \times (-r^2s, 0)$. At this point we let $\alpha$ such that
\begin{equation}\label{al}
\alpha< \text{min} (2s, \text{log}_r(1-\delta))
\end{equation}
With such a choice of $\alpha$, it  follows from \eqref{y1}  by scaling back,  that corresponding to
\begin{equation}\label{c1}
 \rho_{k+1} = \rho_{k} + \delta \bar u_k,
\end{equation} the following estimate holds,
\begin{equation}\label{y2}
|u- \rho_{k+1}|^2 \leq  \tilde  M_k^2 r^{2k \alpha} = \tilde M_{k+1}^2
\end{equation}
in $B_{r^k} \times (-r^{2ks}, 0])$ which  verifies i) and iii) of the hypothesis of the Lemma. Also, we have that
\[
|\rho_{k+1}| + M_k \leq \rho_{k} + \delta \bar u_k + M_{k-1}(1-\delta) \leq |\rho_k| + M_{k-1} + \delta(\bar u_k - M_{k-1}) \leq M
\]
and thus ii) holds as well. Thus we are only left with checking iv). Now note that  from the choice of $\rho_{k+1}$ as  in \eqref{c1}, we have that
\begin{equation}\label{re1}
|\rho_{k+1} - \rho_k| = \delta \bar u_k \leq \delta M_k \leq \delta Mr^{k\alpha}
\end{equation}
From \eqref{re1} and the induction hypothesis, using
\begin{equation}
|\rho_{k+1} -\rho_i| \leq \sum_{l=i+1}^{k+1}  |\rho_{l} -\rho_{l-1}|
\end{equation}
we observe that iv) follows as well with $C= \frac{1}{1-r} \delta$.  Thus the induction step is verified and the conclusion follows. 

\end{proof}


From Lemma \ref{corocontract}, we have that $u$ is H\"older continuous at $(0,0)$ and  since $(0,0)$ is an arbitrary point of the domain,  Theorem \ref{nonlocal} thus follows.

\subsection{Weak harnack for  $(\partial_t - \Delta)^s=H^s$}
We now  prove the  weak harnack inequality for $(\partial_t - \Delta)^s$ as stated in Theorem \ref{fHarnackIntro}.   In order to do so,  similar to that  in \cite{banerjeeGaro}, \cite{NS}  and \cite{ST}, we first introduce the relevant function space which constitutes the natural domain for $(\partial_t- \Delta)^s$, namely
\[
\operatorname{Dom}(H^s) = \{u\in L^2(\R^{n+1})\mid (\xi,\sigma) \to ((2\pi|\xi|)^2 + 2\pi i \sigma)^s\  \hat u(\xi,\sigma)\in L^2(\R^{n+1})\}
\]
We then proceed with the proof of Theorem \ref{fHarnackIntro}.

\begin{proof}[Proof of Theorem \ref{fHarnackIntro}]
Using  the extension approach  for the fractional heat  operator as in \cite{NS} and \cite{ST} which constitutes the parabolic counterpart of the celebrated Caffarelli-Silvestre extension as in \cite{CafSil},  we note that  corresponding to $v$, the function $U$ be defined by
\begin{equation}\label{UU}
U(x,y,t) = \frac{y^{2s}}{2^{2s} \G(s)} \int_0^\infty \tau^{-(1+s)} e^{-\frac{y^2}{4\tau}} e^{-\tau H} v(x,t) d\tau
\end{equation}
 is a weak solution to the following weighted Neumann problem, 
\begin{equation}\label{ext}
\begin{cases}
L_a U= \operatorname{div}(y^a \nabla U) - y^a U_t=0\ \text{in $\{y>0\}$ where $a=1-2s$}
\\
U(\cdot,0, \cdot)= v\ \text{in $L^2(\R^{n+1})$}
\\
- \text{lim}_{y \to 0} y^a U_y= C(s) (\partial_t - \Delta)^s v\ \text{in  $L^{2}(\R^{n+1})$ for some $C(s)>0$.}
\end{cases}
\end{equation}
See for instance Lemma 4.5 in \cite{banerjeeGaro}.
Since $v$ is a supersolution, it follows from \eqref{ext} that if $U$ is evenly reflected across $\{y=0\}$, that the extended  function $V$  is a weak solution to
\begin{equation}\label{ac}
\operatorname{div}(|y|^a \nabla V) - |y|^a V_t \leq 0\ \text{in $B_2 \times (-2, 2) \times (-4, 0]$.}
\end{equation}
We now show the validity of a  mean value inequality for $V$ from which the weak harnack follows as a consequence. 
We adapt a few ideas  from \cite{FG}.  Let $G_a(x,y, t,\xi, \eta, s)$ be the backward heat kernel for $L_a$ as in  Proposition 2.3 in \cite{BDGP} ( see also \cite{G})  with pole at $(\xi, \eta, s) \in \R^n \times \R  \times \R$, i.e. for $t \neq s$, it solves
\[
L_a^* G_a = \operatorname{div}(y^a \nabla G_a) +  y^a \partial_t G_a=0
\]

  For a given $\epsilon, \delta>0$,  Let $\Omega_{\epsilon, \delta}(\xi, \eta, s)$ be the region   defined as
\[
\Omega_{\epsilon, \delta}(\xi, \eta, s)=  \{(x, y, t): G_a(x,y,t, \xi, \eta, s) > 1,  |y| > \epsilon, |t-s| >\delta\}
\]
Since $L_a^* G_a=0$, we have
\begin{align}
& \int_{\Omega_{\epsilon, \delta}} G_a L_a V -  VL_a^* G_a=0
\end{align}
By integrating by parts, we then obtain
\begin{align}\label{l1}
& 2 \int_{\{y= \epsilon\} \cap \Omega_{\epsilon, \delta} } y^a (G_a)_y V -  y^a V_y G_a  - \int_{\{|t-s|=\delta\} \cap \Omega_{\epsilon, \delta}}  y^a G_a V
\\
&+ \int_{\partial \Omega_{\epsilon, \delta} \setminus ( \{|y|=\epsilon\} \cup  \{|t-s|=\delta\})}  y^a G_a <\nabla V, n_{(x,y)}> - V <\nabla G_a, n_{(x,y)}>  -  y^a VG_a n_{t}=0
\notag
\end{align} 
where $n=(n_{(x,y)}, n_s)$ is the normal to the boundary $\partial \Omega_{\epsilon, \delta}$. Here we also used the fact that both $V$ and $G_a$ are symmetric across $\{y=0\}$. Now since $G_a=1$ on $\partial \Omega \setminus  ( \{|y|=\epsilon\} \cup  \{|t-s|=\delta\})$, we  have that
\begin{align}\label{l5}
& \int_{\partial \Omega_{\epsilon, \delta} \setminus ( \{|y|=\epsilon\} \cup  \{|t-s|=\delta\})} y^a G_a <\nabla V, n_{(x,y)}>- y^a V G_a n_{t}=  \int_{\Omega_{\epsilon, \delta}} L_a V 
\\
& + 2 \int_{\{y=\epsilon\}  \cap \Omega_{\epsilon, \delta}} y^a V_y  + \int_{\{|t-s|=\delta\} \cap \Omega_{\epsilon, \delta}} y^a V dX
\notag
\end{align}
From \eqref{l1} and \eqref{l5} we obtain using $L_a V=0$ that the following holds,
\begin{align}\label{l50}
& 2 \int_{\{y= \epsilon\} \cap \Omega_{\epsilon, \delta} } y^a (G_a)_y V -  y^a V_y (G_a -1)  - \int_{\{|t-s|=\delta\} \cap \Omega_{\epsilon, \delta}}  y^a G_a V
\\
&-  \int_{\partial \Omega_{\epsilon, \delta} \setminus ( \{|y|=\epsilon\} \cup  \{|t-s|=\delta\})}  y^a V <\nabla G_a, n_{(x,y)}> + \int_{\{|t-s|=\delta\} \cap \Omega_{\epsilon, \delta}} y^a V dX=0\notag
\end{align}

Now we note that  as $y \to 0$, since $y^a (G_a)_y \to 0$ and $y^a V_y \leq 0 $ at $y=0$, therefore by letting $\epsilon \to 0$ in \eqref{l50} and also by using $G_a -1 \geq 0$ in $\Omega_{\epsilon, \delta}$,  we deduce that the following holds,
\begin{align}\label{l2}
& - \int_{\{|t-s|=\delta\} \cap \Omega_{ \delta}}  y^a G_a V
\\
&- \int_{\partial \Omega_{ \delta} \setminus   \{|t-s|=\delta\}}  y^a V <\nabla G_a, n_{(x,y)}>  + \int_{\{|t-s|=\delta\} \cap \Omega_{ \delta}} y^a V  \leq 0
\notag
\end{align} 
where $\Omega_{\delta}(\xi, \eta, s)= \{(x, y, t): G_a(x,y,t, \xi, \eta, s) > 1,   |t-s| >\delta\}$. Now we claim that $\delta \to 0$, 
\begin{equation}\label{cl1}
\int_{\{|t-s|=\delta\} \cap \Omega_{ \delta}}  y^a G_a V \to V(\xi, \eta, s)
\end{equation}
Now since
\[
\int_{\{|t-s|=\delta\} }  y^a G_a V \to V(\xi, \eta, s)
\]
as $\delta \to 0$,  \eqref{cl1} is  thus equivalent to showing 
\begin{equation}\label{cl2}
\int_{\{|t-s|=\delta\} \cap \{G_a <1\}}  y^a G_a V  \to 0
\end{equation}
Now \eqref{cl2}  can be seen  from a change of variable formula of the type
\[
|t-s|^{1/2} Z + (\xi, \eta)= X
\]
where $X=(x,y)$,  using the explicit representation of $G_a$ as in (2.10) in \cite{BDGP} and also by using the asymptotics of the Bessel function $I_{\frac{a-1}{2}}$ as in   (6.19) in \cite{BDGP}. Now \eqref{l2}, \eqref{cl1}  coupled with the fact that
\[
\int_{\{|t-s|=\delta\} \cap \Omega_{ \delta}} y^a V \to 0\ \text{as $\delta \to 0$}
\]

implies
\begin{align}\label{l4}
& - V(\xi, \eta, s)
\\
&- \int_{\partial \Omega }  y^a V \left<\nabla G_a, n_{(x,y)}\right>  \leq 0
\notag
\end{align}
where $\Omega= \{G_a >1 \}$. Now  since $n_{x,y}= -\frac{\nabla G}{|\nabla_{(x,y,t)} G|}$, we obtain from \eqref{l4} that the following holds, 
\begin{align}\label{me1}
& V(\xi, \eta, s) \geq \int_{\partial \Omega} y^a  V \frac{|\nabla G_a|^2}{|\nabla_{(x,y,t)} G_a|}
\end{align}
More generally, if $\Omega_r(\xi, \eta, s)= \{G_a (x,y,t, \xi, \eta, s) > r\}$, then we have that
\[
V(\xi, \eta, s) \geq \int_{\partial \Omega_r} y^a  V \frac{|\nabla G_a|^2}{|\nabla_{(x,y,t)} G_a|}
\]
Then by using the Coarea formula, we obtain for all $r>0$, that the  following mean value inequality holds, 
\begin{equation}\label{me0}
V(\xi,\eta, s) \geq \int_{\Omega_r(\xi, \eta, s)} y^a V |\nabla G_a|^2
\end{equation}
Now since $V(\cdot, \eta) \to v$ in $L^{2}$ as $\eta \to 0$, therefore upto a subsequence $\eta_k \to 0$, we have that $V(\cdot, \eta_k) \to v$ pointwise a.e. Therefore from \eqref{me0} it follows for  a.e. $(\xi, s)$,
\begin{equation}\label{m1}
v(\xi, s) \geq \int_{\Omega_r(\xi, 0, s)} y^a V |\nabla G_a(\cdot, \xi, 0, s)|^2 =   \int_{\Omega_r(\xi, 0, s)} y^a V \frac{|X- (\xi, 0)|^2}{2 |t-s|^2}  G_a dXdt
\end{equation}
Now let $r=1$ in \eqref{m1}. Then from the representation of $U$( and consequently that of $V$)  in terms of $v$ as in  \eqref{UU} and the fact that $v \geq 0$,  we observe that  the following holds,

\begin{equation}\label{lb}
 \int_{\Omega_1(\xi, 0, s)} y^a V \frac{|X- (\xi, 0)|^2}{2 |t-s|^2}  G_a dXdt \geq C(n, s) \int_{ \{(x,t): |x- \xi| < 1/2,  1/4< s-t < 1/2 \}}  v(x,t) dx dt
 \end{equation}
Using \eqref{lb} in \eqref{m1} we finally have  for a.e. $(\xi, s)$
 \begin{equation}\label{wk}
 v(\xi, s) \geq C \int_{ \{(x,t): |x- \xi| < 1/2,  1/4< s-t < 1/2 \}}  v(x,t) dx dt
 \end{equation}
 from which the weak harnack estimate as claimed in \eqref{f2} follows. This finishes the proof of the Theorem.
 \end{proof}
 We  end our discussion with the following remark.
 
 \begin{remark}
 We note that although the equation in \eqref{ac} constitutes a prototypical example of a more general class of equations studied in \cite{CSe},  a weak harnack estimate of the type \eqref{f2} with $L^{1}$ norm on the left hand side is not explicitly proven in \cite{CSe}. Moreover, we also believe that mean value type inequality that we establish for generic supersolutions to the extension operator $L_a$ as an intermediate step   may find other applications. 
 
 \end{remark}


\end{document}